\documentclass[12pt]{amsart}

\usepackage{amsmath,amssymb,mathrsfs}
\usepackage{a4wide}
\usepackage{amsthm} 

\usepackage{srcltx} 

 \usepackage[dvips]{graphicx,color}
\usepackage{amsopn,cite,mathrsfs}
\usepackage{amsfonts}
\usepackage{eucal}

\usepackage{times}

\numberwithin{equation}{section}

\newtheorem{thm}{Theorem}[section]
\newtheorem{lemma}[thm]{Lemma}

\newcommand{\D}{{\mathbb D}}
\newcommand{\C}{{\mathbb C}}
\newcommand{\Cn}{{\mathbb C}^n}
\newcommand{\T}{{\mathbb T}}

\def\N{\ensuremath{\mathbb{N}}}

\def\cS{\mathcal S}
\def\etSq{\ensuremath{{\mathcal{S}}^{\ast}_{q}}}
\def\cM{\ensuremath{\mathcal{M}}}

\def\Oms{\D_{s}}
\def\Omts{\D_{s}}
\def\cOmts{\overline{\D}_{ s}}

\def\C{\ensuremath{\mathbb{C}}}

\def\N{\ensuremath{\mathbb{N}_0}}

\begin{document}

\title[ The Polyanalytic   Reproducing Kernels]
{The Polyanalytic   Reproducing Kernels}

\author{Hicham Hachadi and El Hassan Youssfi}


\address{Hachadi: Coll\`ege Benzekri\\
         Mehdya,  K\'enitra, Maroc}
    \email{hachadi.hi82@gmail.com}   
      
\address{Youssfi: I2M, U.M.R. C.N.R.S. 7373, CMI\\
         Universit\'e d'Aix-Marseille\\
         39 Rue F-Joliot-Curie\\
         13453 Marseille Cedex 13, France}
\email{el-hassan.youssfi@univ-amu.fr}



\begin{abstract} Let $\nu$ be a rotation invariant Borel probability measure on the complex plane having moments of all orders. Given a positive integer $q$,  it is proved that the space of $\nu$-square integrable $q$-analytic functions is the closure of   $q$-analytic polynomials, and in particular it is a Hilbert space.  We establish a general formula for the corresponding polyanalytic reproducing kernel.  New examples are given and all known examples, including those of the analytic case are covered.
In particular,  weighted Bergman and Fock type spaces of polyanalytic functions are introduced. Our results have a higher dimensional generalization for measure on $\C^p$ which are in rotation invariant with respect to each coordinate.

 \end{abstract}

\maketitle

\section{Introduction}
 In their recent work A. Haimi and H. Hedenmalm \cite{HH1}--\cite{HH2}, established asymptotics for the Bergman-Fock type space of  polyanalytic functions with respect to a given weight and mentioned that in general finding explicit formula for these kernels is difficult  (\cite{HH2}, p. 4668).  This problem was also addressed by D. Alpay (\cite{Al}, p. 479). The main goal of this paper is to answer this question in the  general context of  rotation invariant Borel probability measure on the complex plane having moments of all orders.  More precisely, given a positive integer $q$,  we shall establish the formula for the reproducing kernel for Hilbert spaces  of square $q$-analytic functions with respect to a rotation invariant Borel probability measure.  The reduction of our formula to the unit disc $\D$ gives an explicit formula for the weighted Bergman spaces of polyanalytic functions on $\D$, which in turn,  which reduces to the result of Koshelv \cite{Ko} when the weight is trivial.   We point out that the result of Koshelev is proved by a very specific method based on integration by parts which does not work for the weighted case.  Other applications are given to provide new results on other Bargmann-Fock type  spaces of polyanalytic functions  on $\C$ and related projections.

We recall that a function $f\left(z\right)$ is called a polyanalytic function of order $q$  (or just  $q$-analytic) in the domain $\Omega\subseteq\mathbb{C}$ if in this domain it satisfies the generalized Cauchy-Riemann equation
\begin{equation}
\frac{\partial^{q}f}{\partial {\bar z}^{q}}=0.
\end{equation} 

Polyanalytic functions inherit some of the properties of analytic functions and the simplest case is the so-called bianalytic functions. However, as in the theory of several complex variables, many of the properties break down once we leave the analytic setting. They are naturally related to polyharmonic functions see \cite{Ba}, \cite{Ha} and \cite{Re1} for further results.    
   
The properties of these functions was studied by several authors see Balk and Zuev \cite{BZ}, Balk \cite{BM} and Dzhuraev \cite{D} and the references therein.  It is well known that any $q$-analytic function in the domain $\Omega$ can be uniquely expressed as 
\begin{equation}\label{almansi}
f\left(z\right)=\underset{j=0}{\overset{q-1}{\sum}}\overline{z}^{j}\phi_{j}\left(z\right).
\end{equation}
where the $\phi_{j}\left(z\right)$ are holomorphic in $\Omega$. This representation was used to study the boundary behavior and integral representation of polyanalytic functions.
Hilbert spaces of polyanalytic functions and related projections  were considered for the case of the unit disc by Koshelev \cite{Ko} and later by Vasin \cite{Va} and A. K. Ramazanov \cite{Ra1} and \cite{Ra2}. In the latter reference   a representation of the space of polyanalytic functions as direct sum of orthogonal subspaces is given and applied to rational approximation.  The case of the Bargmann-Fock space of polyanalytic functions was studied by   N. L. Vasilevski \cite{Vas1}-\cite{Vas2} and later by L. D. Abreu \cite{Ab1}-\cite{Ab2} in connection with Gabor and time-frequency analysis.  A deep study of the general case of weighted Bargmann-Fock space of polyanalytic functions was considered by  A. Haimi and H. Hedenmalm  \cite{HH1}-\cite{HH2}, where they obtain the asymptotic expansion of the polyanalytic Bergman kernel  as well as  the asymptotic behavior of the generating kernel and the asymptotic in the bulk for the $q$-analytic Bergman spaces in the setting of the weights $e^{-2mQ}$ (see \cite{HH2}). Their approach relies on the study   of  polyanalytic Ginibre ensembles and  appeals to  the connection with random normal matrix theory and Landau levels. 

Polyanalytic functions of several variables were considered by Avanissian and Traor\'e \cite{AT1} and \cite{AT2}.  They are defined in an anlogous way. Namely,  a function $f\left(z\right)$ is called a polyanalytic function of order $q =(q_1, \cdots, q_p) \in \N^p$ (or just $q$-analytic) in the domain $\Omega\subset \C^p$ if in this domain it satisfies the generalized Cauchy-Riemann equation
\begin{equation}
\frac{\partial^{q_1 + \cdots q_p }f}{\partial {\bar z}_{1}^{q_1}\cdots \partial {\bar z}_{p}^{q_p}} =0.
\end{equation} 

 They can be uniquely expressed as 
\begin{equation}\label{almansi}
f\left(z\right)=\underset{j=(0, \cdots, 0)}{\overset{(q_1-1, \ldots, q_p -1)}{\sum}}\overline{z}^{j}\phi_{j}\left(z\right)
\end{equation}
where the $\phi_{j}\left(z\right)$ are holomorphic in $\Omega$ where for 
 $j=(j_1, \cdots, j_p),  k= (k_1, \ldots, k_p )  \in \N^p$  and  $z= (z_{1}, \cdots z_{p}) \in \C^p$,  the inequality $j\leq k$ means that $j_l \leq k_l $ for all $l=1, \cdots, p$ and  
 $z^j := z_{1}^{j_{1}} \cdots z_{p}^{j_{p}}$.

However, there are very few results are available in this case.

\section{Statements of the main results}

In this section we will state the main results in the one dimensional case. The higher dimensional analogs will be stated at the end of the paper. 

The setting is the following.   We recall that a  sequence  $s= (s_d), d \in \N,$  is said to be  a Stieltjes moment sequence if it has  the form
$$s_d = \int_{0}^{+\infty} t^d d \mu(t),$$
where $\mu$ is a non-negative measure on $[0, +\infty[$, called a representing measure for  $s.$  These sequences have been  characterized 
by Stieltjes \cite{St} in terms of some positive definiteness conditions. We denote by $\cS$ the set of such sequences and if $s \in \cS$ we let $\cM(s)$ the convex cone of the representing measures of $s.$  It follows from the above integral representation that each  $s \in \cS$ is either non-vanishing; that is, $s_d > 0$ for all $d$, or else $s_d = \delta_{0}$ for all $d$. We denote by $\etSq$ the set of all non-vanishing elements of $\cS$ having a representing measure $\mu$ with support containing at least $q$ strictly positive elements. Fix  an element $s= (s_d) \in \etSq$ and  let $\mu\in \cM(s)$. It is known \cite{BY1} that the sequence 
${(s_{d}}^{\frac{1}{2d}})$ converges to limit $R_s \in ]0, +\infty],$  where
$R_s $ is the supremum over all $t>0$ such that $t$ is in the support of $\mu.$
We denote by $\Oms$ the disc  
in $\C$ centered at the origin with radius $R_s$ with the understanding that $\Oms = \C$ when $R_s =+\infty.$

 For each pair of  non-negative integers $(d, n)$ such that $n\leq q-1$,  let ${\mathcal P}_n(\mu)$  be the  subspace of the Hilbert $L^{2}(x^d d\mu(x))$ consisting  of  all  polynomials with degree at most $n$ furnished with the real inner product
$$\langle f, g \rangle :=  \int_{0}^{+\infty} f(x) g(x) x^d d\mu(x), \ \ f, g \in {\mathcal P}_n(\mu)$$
and denote by  $Q_{d, n } : (0, +\infty) \times (0, +\infty) \to \C $ the corresponding reproducing kernel. 

 Consider the  following function
\begin{eqnarray}\label{Fsq} F_{q, s} (\lambda,  x, y ) := \sum_{d=0}^{+\infty} \lambda^d Q_{d, q-1}(x, y) +  \sum_{d=1}^{q-1}  {\bar  \lambda}^d  Q_{d, q-1-d}(x, y) . 
\end{eqnarray}
where $\lambda$ is a complex number and  $(x, y)  \in [0, +\infty[\times [0, +\infty[$.
Our first result is the following:
\newtheorem*{thA}{\bf Theorem A}  
\begin{thA}{\it   For all fixed non-negative real numbers $x$ and $y$, the  series $ \lambda \mapsto F_{q, s} (\lambda , x, y ) $ converges uniformly on compact subsets of the disc  centered at $0$ with radius $R_s^2$. }
\end{thA}

 Next, let  $\mu \in \cM(s)$ and  denote by $\nu $ denote the  image measure on $\C$  of $\mu \otimes \sigma $ under the map $(t, \xi) \mapsto \sqrt{t}\xi$ from $[0, +\infty[\times \T$ onto $\C$,  where $\sigma$  is the rotation invariant probability measure on the unit circle $\T$ in $\C$. Then $\nu$ is rotation invariant. Conversely, it is known \cite{BT} that any  rotation invariant  Borel probabiltiy  $\nu$ on $\C$ is of this form. 
    Since $\mu$ is supported in the interval $[0, +R_{s}]$, it follows that the support of $\nu$ is contained in closure $\cOmts$  of the open disc
 $\Omts.$

  We consider the Hilbert space $L^2(\nu)$ of square integrable
complex-valued functions  in $\cOmts$ with respect to the measure
$\nu$.
We denote by ${\mathcal A}_{\nu, q}^{2}$   the 
 space  of those $q$-analytic functions on $\Oms$ which are square integrable with respect to $\nu.$
The natural inner product inherited from that of $L^2(\nu)$ turns ${\mathcal A}^{2}_{\nu, q}$ into a pre-Hilbert space.
  We are now prepared to state our second main result.
\newtheorem*{thB}{\bf Theorem B}  
\begin{thB}{\it   The space  ${\mathcal A}_{\nu,q}^2$ is a Hilbert space which coincides with the closure  of the $q$-analytic polynomials in  $L^2(\nu)$. Moreover, 
  for each  set compact $K \subset  \Oms$ we have that
 $$ {\sup}_{z \in K} {|f(z)|}
\leq C \parallel f \parallel _{L^2(\nu)}$$
 for all    $q$-analytic  polynomials $f \in L^2(\nu) $, where
 $$C=C(K) := {\sup}_{z\in K} \sqrt{F_{q, s} (|z|^{2},  |z|^{2}, |z|^{2} )}.$$
  Furthermore, 
the reproducing kernel  of   ${\mathcal A}_{\nu, q}^2$    is given by
$$R_{\nu, q}(z, w) =  F_{q, s} (z\bar w, |z|^{2}, |w|^{2}),  \ \ z, w \in \Oms.$$ }
\end{thB}
\newtheorem*{RemC}{\bf Remark C }  
\begin{RemC}{\rm  When the measure $\mu$ has a finite support and the number points of the support of $\mu$ is $q$, then
Theorem B gives à Cauchy type formula for polyanalytic functions. Such an example can be obtained using the Kroutchouk measure and related orthogonal polynomials.}
\end{RemC}
A first  application of our results provides the weighted polyanalytic Bergman kernel of the unit disc $\{z\in \C : |z| <1\}$. More precisely,  for $\alpha >-1,$ we consider the  space ${\mathcal A}_{\alpha, q}^2$ of all square integrable of  $q$-analytic functions  with respect to the  measure $ d\nu_\alpha(z) : (1-|z|^2)^\alpha \frac{dA(z)}{\pi}, $ where $dA(z)$ is the Lebesgue measure on $\D$. We will prove the following
\newtheorem*{thD}{\bf Theorem D}  
\begin{thD}{\it   The space ${\mathcal A}_{\alpha, q}^2$  is a Hilbert space which coincides with the closure  of the $q$-analytic polynomials in  $L^2(\nu_\alpha )$ and its  reproducing kernel    is given by
\begin{eqnarray*}  K_{\alpha, q}(z, w)  & =  
 q  \binom{\alpha + q-1}{\alpha}  \frac{(1- \bar z w)^{q-1}}{(1- z\bar w)^{\alpha +q+1}}  \sum_{j=0}^{q-1} (-1)^{j} \binom{q-1}{j} \binom{\alpha +q+j}{\alpha + q-1}    \frac{ |z-w|^{2j} }{ |1 - z\bar w|^{2j}  }.
\end{eqnarray*} 
for all $z, w \in \D.$ }
\end{thD}

We point out that when $\alpha =0,$ this result was established by Koshelev \cite{Ko} by 
different method  limited to the case  $\alpha =0,$ but does not work for  $\alpha \not =0.$ 

A third application of our results provides the weighted polyanalytic Bergman kernel for the weighted Fock space. Namely, let $\alpha >0,$ and denote by
${\mathcal F}_{\alpha, q}(\C)$  the  space ${\mathcal A}_{\alpha, q}^2$ of all square integrable of  $q$-analytic functions  with respect to the  measure $ d\nu_\alpha(z) : |z|^{2\alpha} e^{- |z|^2} \frac{dA(z)}{\pi}, $ where $dA(z)$ is the Lebesgue measure on $\C$ We will establish the following
\newtheorem*{thE}{\bf Theorem E}  
\begin{thE}{\it   The space ${\mathcal F}_{\alpha, q}$  is Hilbert space which coincides with the closure  of the $q$-analytic polynomials in  $L^2(\nu_\alpha )$ and its  reproducing kernel    is given by
\begin{eqnarray*}  K_{\alpha, q}(z, w)  & =  e^{z\bar w}L_{q-1}^{\alpha +1}\left(|z- w|^2\right), \ \  \ \text{for all} \  z, w \in \C
\end{eqnarray*} 
 where  $L_{q-1}^{\alpha}$ is classical weighted Laguerre polynomials  of degree $q-1$ and weight $\alpha
$.  }
\end{thE}

We point out that when $\alpha =0,$ this result was established by Haimi and Hedenmalm \cite{HH1} using a
different method which  does not go trough for   $\alpha \not =0.$

\section{Preliminary results} We collect a few preliminary results from the refrences \cite{DX} or \cite{I}. 
Let $s=(s_n)$  be a Stieltjes moment sequence  and   $\mu \in \cS$ representing measure  of $s$. We assume in this section  that the support of  $\mu$  has $N(\mu) \geq q$ elements.  For each   non-negative integers $n \leq N(\mu) -1$  set

 $$D_{\mu, n} := 
 \left |
\begin{matrix}
 s_{0} & s_{1} & ... & s_{n} \\
 s_{1}  & s_{2} & ... & s_{n +1} \\
\vdots & \vdots    & \cdots & \vdots \\
s_{n} & s_{n+1} &  \cdots  & s_{2n} \\ 
\end{matrix}
\right|
$$
and for $x\in \C, $ let 
$$D_{\mu, n}(x) := 
 \left |
\begin{matrix}
 s_{0} & s_{1} & ... & s_{n} \\
 s_{1}  & s_{2} & ... & s_{n +1} \\
\vdots & \vdots    & \cdots & \vdots \\
s_{n-1} & s_{ n} &  \cdots  & s_{2n-1} \\ 
1 & x &  \cdots  & x^n\\ 
\end{matrix}
\right|
$$
It is well-known that the sequence $(P_{\mu, n} )_{n=0}^{N(\mu) - 1}$ of  orthogonal  polynomials with respect to the measure $d\mu(x)$ is given by 
\begin{eqnarray}\label{orthog-poly} P_{\mu,  n}(x)= \frac{D_{\mu, n}(x)}{\sqrt{D_{\mu, n-1} D_{\mu, n}}}.
\end{eqnarray}
so that the reproducing kernel $Q_{\mu, n} $ is given by 
\begin{eqnarray}\label{rep-kern-poly}Q_{\mu, n} (x,y) = \sum_{j=0}^n \frac{D_{\mu, j}(x)D_{\mu, j}(y)}{D_{\mu, j-1} D_{\mu, j}}.
\end{eqnarray}

We recall the following classical theorem of Heine a proof of which can be found in \cite{DX}

\begin{lemma}\label{Heine} The  determinants  $D_{ n}$ and $D_{ n}(x)$ have the integral representations
\begin{eqnarray}\label{integ-Dm}  D_{n} & = &\frac{1}{(n+1)!} \int_{[0,+\infty[^{n+1}}   \prod_{1\leq j<k\leq n+1}(x_j-x_k)^2   d\mu(x_1) \cdots d\mu(x_{n+1}) \\
\label{integ-Dx} D_{n}(x) & = & \frac{1}{n!}  \int_{[0,+\infty[^n} \prod_{i=1}^n (x-x_i)  \prod_{1\leq j<k\leq n}(x_j-x_k)^2   d\mu(x_1) \cdots d\mu(x_n)
\end{eqnarray} 
 \end{lemma}

In what follows we shall  fix  the measure $ \mu,$ and for each $d \in \N,$ we consider determinants and orthogonal polynomials with respect to the measure $x^d d\mu(x).$ Then  we simply  set
\begin{eqnarray}
D_{x^d \mu(x) , n} := D_{d, n} \ \ \text{and} \ \ D_{x^d \mu(x) , n}(x) := D_{d, n}(x).
\end{eqnarray} 
The   sequence of  orthogonal  polynomials with respect to the measure $x^{d}d\mu(x)$ will be then denote by $(P_{d, m} )_{m=0}^{N(\mu) -1}$  and it is given by 
\begin{eqnarray}\label{orthog-poly} P_{d, n}(x)= \frac{D_{d, n}(x)}{\sqrt{D_{d, n-1} D_{d, n}}}.
\end{eqnarray}
so that the  corresponding reproducing kernel $Q_{d, n} $ is given by 
\begin{eqnarray}\label{rep-kern-poly}Q_{d, n} (x,y) = \sum_{j=0}^n \frac{D_{d, j}(x)D_{d, j}(y)}{D_{d, j-1} D_{d, j}}.
\end{eqnarray}

\begin{lemma}\label{infinite-support}  Suppose that the support of $\mu$ is unbounded. Then for any positive integer $n$ and  $x\in [0,+\infty[$, there exist $t_x > 0$  and a constant $C_x>0$ such that
\begin{eqnarray}  |P_{d, n}(x)| \mu\left([t_x, +\infty[\right) & \leq  \frac{C_x}{t^{d/2}} 
\end{eqnarray} 
for all $t\geq t_x.$
 \end{lemma}
\begin{proof} In view of Lemma \ref{Heine} by Cauchy-Schwarz inequality we see that
\begin{eqnarray*}\left | D_{d, n}(x)\right|^2   \leq D_{d, n-1}  \int_{[0,+\infty[^n} \prod_{i=1}^n (x-x_i)^2  \prod_{1\leq j<k\leq n}(x_j-x_k)^2   (x_1 \cdots x_n)^d d\mu(x_1) \cdots d\mu(x_n)
\end{eqnarray*} 
 Since the degree  $n$ of the polynomial $D_{d, n}(x)$ is positive, there is $t_x > 0$ such that 
$$|D_{d, n}(x) | \leq |D_{d, n}(t) |, \ \ \text{for all} \ t\geq t_x$$
and thus by Lemma \ref{Heine} we get
\begin{eqnarray*} \frac{\left | D_{d, n}(x)\right|^2}{D_{d, n-1}}  \int_{t_x}^{+\infty}  x_{n+1}^d d\mu(x_{n+1}) & \leq  & \int_{t_x}^{+\infty}  \left | D_{d, n}(x_{n+1})\right|^2 x_{n+1}^d d\mu(x_{n+1})  \\
& \leq  & (n+1)!   D_{d, n} 
\end{eqnarray*} 
Taking $C= \sqrt{(n+1)!}$  and using (\ref{orthog-poly}) completes the proof.
\end{proof}
\section{Orthogonal polynomials with respect to rotation invariant measures}
Throughout this section, fix  an element $s= (s_d) \in \etSq$ and  let $\mu\in \cM(s)$. 
 For each pair of  non-negative integers $(d, n),$ with $d$ arbitrary and $n\leq q-1$,  let $( P_{d, k}), k\in\{0, \cdots, n\}$  be a sequence of orthonormal polynomials of
   Hilbert  space ${\mathcal P}_n(x^d\mu)$  equipped the $L^{2}(x^d d\mu(x))$ inner product. 
For all integers  $m, n \in \N $, set 
\begin{eqnarray} 
 m \wedge n := \min(m, n).\end{eqnarray}
 and  for all $z = r\xi \in \C, r \geq 0, |\xi| =1,$ 
 \begin{eqnarray}
   H_{m, n }(z)  :  = r^{ |m - n| } \xi^{m} {\bar \xi}^{ n}P_{|m - n|, m \wedge n}(r^{2}).
    \end{eqnarray}

\begin{lemma}\label{orthog}
The family $(H_{m, n })$ forms  an  orthogonal  system   in $L^2(\nu).$ 
\end{lemma}
\begin{proof}
Let  $(m, n), (m', n') \in \N^2.$  We first observe that for $m +  n' = m' +  n,$ we have

$$\aligned
\int_{0}^{+\infty} H_{m, n }(r^{1/2}) \overline{ H_{m', n' }(r^{1/2}) } d\mu(r)  
&  =     \int_{0}^{+\infty} r^{|m- n | } P_{|m - n|, m \wedge n}(r)
P_{|m - n|, m' \wedge n'}(r)
d\mu(r) \\
& =  \delta_{m\wedge n , m' \wedge n' } 
 \endaligned
 $$

By the change of variables formula, we see that
$$\aligned
\int_{\Oms} H_{m, n }(z) \overline{ H_{m', n' }(z) } d\nu(z)  &  = \int_{0}^{+\infty} \int_{\T} 
   H_{m, n }(r^{1/2}\xi) \overline{ H_{m', n' }(r^{1/2}\xi) }   d\sigma(\xi) d\mu(r) \\
&  =     \int_{0}^{+\infty} 
   H_{m, n }(r^{1/2}) \overline{ H_{m', n' }(r^{1/2}) } d\mu(r)  \int_{\T}  \xi^{m+ n'} 
 \overline{\xi^{m' +  n}} d\sigma(\xi) \\
& =  \delta_{m +  n', m' +  n} \delta_{m\wedge n , m' \wedge n' }   \\
& =  \delta_{(m,  n), (m',  n')}. 
 \endaligned
 $$

This completes the proof.
\end{proof}
\begin{lemma}\label{radius1} Let $n $ and $d $  be positive integers such that $n\leq q-1.$  Consider a polynomial $f$ in $n$-variables  such that 
$ f(x_1, \cdots, x_n) > 0$ for all pairwise distinct elements $x_1, \cdots, x_n$ of $[0, +\infty[,$
   and  set 
\begin{eqnarray} \gamma_{d, n}(f) := \int_{0}^{+\infty} \cdots \int_{0}^{+\infty}  f(x_1, \cdots, x_n)  (x_1 \cdots x_n)^d d\mu(x_1) \cdots d\mu(x_n)
\end{eqnarray} 
Then
 \begin{eqnarray} 
 \lim_{d \to +\infty}  \frac{ \gamma_{d +1, n}(f)}{ \gamma_{d, n}(f)} =  \left(\lim_{d \to +\infty}  \frac{ s_{d+1}}{ s_{d}}\right)^n
  \end{eqnarray} 
 \end{lemma}
\begin{proof} Let $\eta $ be the  image of the measure on $[0, +\infty[$ of $ f(x_1, \cdots, x_n)  d\mu(x_1) \cdots d\mu(x_n)$ under the map $(x_1, \cdots, x_n)   \mapsto x_1 \cdots x_n$.
Then 
\begin{eqnarray}\gamma_{d , n}(f) := \int_{0}^{+\infty} x^d  d\eta(x)
 \end{eqnarray} 
so that by \cite{BY1} we see that
\begin{eqnarray} 
 \lim_{d \to +\infty}  \frac{ \gamma_{d+1, n}(f)}{ \gamma_{d, n}(f)} =  R
  \end{eqnarray} 
  where $R $ is the supremum over all $t>0$ such that $t$ is in the support of $\eta.$ Moreover, it can be easily checked that if $t>0$, then $t$ is in the support of $\eta$ if and only if $t^{\frac{1}{n}}$ is in the support of $\mu.$
 This completes the proof.
\end{proof}
Now we prove Theorem A

\newtheorem*{proofA}{Proof of Theorem A}
\begin{proofA} {\rm 
We only need prove that for all non-negative real numbers $x$ and $y$, 
  the series
 \begin{eqnarray}\label{Ssq} S_{s, q}( \lambda)=  \sum_{n =1}^{q-1}  \sum_{m=0}^{+\infty} \lambda^m P_{m, n }(x)\overline{P_{m, n }(y)}, 
 \end{eqnarray} 
 converges uniformly on compact sets of $\Oms.$
 We shall distinguish two cases. First, assume that the support of $\mu$ is bounded; that is $R_s$ is finite.
 In view (\ref{orthog-poly}), the latter series can be written if the form
$$ S_{s, q}( \lambda) := \sum_{n =1}^{q-1}  \sum_{m=0}^{+\infty}  \lambda^m  \frac{D_{m, n}(x) D_{m, n}(y)}{D_{m, n-1} D_{m, n} }.$$
Using the integral expressions (\ref{integ-Dm}) and (\ref{integ-Dx}) with respect to the measure $r^m d\mu(r)$ 
instead of $d\mu(r)$,  we see that 
$   D_{m, n}(x )$   is  a finite sums of terms of the form $x^j  \gamma_{m , n}(f) $ where $j \in \N$ and $f$ is a function of the form
\begin{eqnarray} \label{f-form} f(x) = x_{1}^{k_1} \cdots x_{n}^{k_n}\prod_{1\leq j<k\leq n}(x_j-x_k)^2.
\end{eqnarray}
The same holds for $   D_{m, n}(y )$ with $y$ instead of $x$. Finally, we observe that
$$D_{m, n}  =  \gamma_{m , n}(g), \ \ \text{and} \ \ D_{m, n-1}  =  \gamma_{m , n-1}(g)$$
where
\begin{eqnarray} \label{g-form} g(x):= \prod_{1\leq j<k\leq n}(x_j-x_k)^2.
\end{eqnarray}
Therefore the series $S_{s, q}$ is a linear combination of series of the form
$$ S_{s, q, j, l}(  \lambda) := \sum_{n =1}^{q-1}  \sum_{m=0}^{+\infty}  \lambda^m x^{j} y^{l}
 \frac{\gamma_{m , n}(f)\gamma_{m , n}(h)}{ \gamma_{m , n}(g)\gamma_{m , n-1}(g)    }$$
 where $f$ and $h$ are of the form (\ref{f-form}) and $g$ is given by (\ref{f-form}).
 Appealing to Lemma \ref{radius1} and using D'Alembert's rule yields that the series $ S_{s, q, j, l}( \lambda) $ converges as long as $| \lambda| < R_s$. From this it is also clear that the series converges uniformly on compact sets of $\Oms.$ 
 
 Next, suppose that $R_s =+\infty.$  Let $x, y$ be arbitrary non-negative real  numbers. Then by Lemma \ref{infinite-support}, there $t_{x, y}$ such that 
 \begin{eqnarray*}\left | P_{m, n }(x) P_{m, n }(y) \right| \leq \frac{(n+1)!}{ t^m}, 
 \end{eqnarray*}
 for all $t\geq t_{x, y}$. This proves that the series (\ref{Ssq})
 convergence absolutely.
 
 Finally, the inequality in Theorem A follows by Cauchy-Schwarz inequality. The remaining equality in the theorem is straightforward.
 This completes the proof.}
\qed \end{proofA}

Next, we denote by ${\mathcal A}^2(s)$   the 
subspace of $L^2(\nu)$ consisting of all functions of the form 
 $$f(z) = \sum_{n=0}^{q-1}  \sum_{|m|=0}^{+\infty} a_{m, n} H_{m, n}(z)$$ on  $\Oms$ that satisfy 
$$    \sum_{n =1}^{q-1}  \sum_{m=0}^{+\infty}   \left|a_{m, n} \right|^2  < +\infty.$$
We equip the space  ${\mathcal A}^2(s)$ with the natural inner product 
\begin{eqnarray}  \langle f, g \rangle_{s} := \sum_{n=0}^{q-1}  \sum_{m=0}^{+\infty} a_{m, n}\overline{ b_{m, n}},
\end{eqnarray}
for all  members $f(z) = \sum_{n=0}^{q-1}  \sum_{m=0}^{+\infty} a_{m, n} H_{m, n}(z)$ and $  g(z) = \sum_{n=0}^{q-1}  \sum_{m=0}^{+\infty} b_{m, n} H_{m, n}(z)$ of ${\mathcal A}^2(s)$.
It is standard that this is a Hilbert space which contains all $q$-analytic polynomials, which is contained in $L^2(\nu)$ and its inner product coincides with the scalar product inherited from the  scalar product of $L^2(\nu)$.
Indeed, we have

\begin{thm}\label{kernel}  The space ${\mathcal A}^2(s)$  consists of   $q$-analytic functions and its reproducing
$K_{s,q}$  kernel is given by
 \begin{eqnarray*}  K_{s, q}( z, w)=  F_{s, q}(z\bar w, |z|^{2}, |w|^{2}),  \ z, w \in \Oms, \end{eqnarray*} 
 where $F_{s, q}$ is the function defined by (\ref{Fsq}).
\end{thm}
\begin{proof}  By virtue of  Theorem A, 
the series
\begin{eqnarray} \label{ksq} K_{s, q}( z, w)=  \sum_{n =0}^{q-1}  \sum_{m=0}^{+\infty} H_{m, n }(z)\overline{H_{m, n }(w)}. 
 \end{eqnarray}
 converges uniformly for $z\bar w $ lying in a  compact subset of  $\Oms.$
  Since the system   $(H_{m, n }), m, n \in \N, n \leq q-1$ forms  an orthonormal basis of ${\mathcal A}^2(s)$,  a little computing shows that
 $$\aligned K_{s, q}( z, w) & =  \sum_{n =0}^{q-1}  \sum_{m=0}^{+\infty} H_{m, n }(z)\overline{H_{m, n }(w)} \\
& = \sum_{n =0}^{q-1}  \sum_{m=0}^{+\infty} (z\bar w)^m P_{m, n }(|z|^{2})P_{m, n }(|w|^{2})  + \sum_{n =1}^{q-1}  \sum_{m=0}^{n -1} (z\bar w)^{m} P_{n -m, m }(|z|^{2})P_{m, n }(|w|^{2}) \endaligned
$$
When $q=1,$ we are in the analytic case. Since  $P_{m, 0 }$ is constant, the latter sum gives
$$K_{s, 1}( z, w) =   \sum_{m=0}^{+\infty} (z\bar w)^{m} P_{m, 0 }P_{m, 0 }.$$  
However, , when $q\geq 2,$  we have
$$\aligned K_{s, q}( z, w)
& =  \sum_{m=0}^{+\infty} (z\bar w)^{m}  \sum_{n =0}^{q-1} P_{m, n }(|z|^{2})P_{m, n }(|w|^{2})  +   \sum_{n =1}^{q-1}  \sum_{m=0}^{n -1} (\bar z w)^{n-m} P_{n -m, m }(|z|^{2})P_{n -m, m  }(|w|^{2}) \\
&=  F_{s, q}(z\bar w, |z|^{2}, |w|^{2}).
 \endaligned
 $$
 Now each element  $f$ of  ${\mathcal A}^2(s)$  admits a unique representation
$$f(z) = \sum_{n =0}^{q-1}  \sum_{m=0}^{+\infty} a_{m, n }H_{m, n }(w).$$
By Cauchy-Schwarz inequality, it follows that this series converges uniformly on compact sets of $\Oms$ and hence it defines a $q$-analytic function.
Moreover, it can be easily checked
$$f(z) = \langle f,  K_{s, q}( \cdot, z) \rangle.$$
for all $z\in \Oms.$
 This completes the proof.

\end{proof}

Now we are ready to  prove Theorem B
\newtheorem*{proofB}{Proof of Theorem B}
\begin{proofB} {\rm 
 It suffices to show that each $q$-analytic function which belongs to $L^2(\nu)$ is an element of the space ${\mathcal A}^2(s)$. Now 
 let $f$ be   $q$-analytic function which belongs to $L^2(\nu)$. By~\eqref{almansi} we know 
that $f$ has a unique representation of the form
\begin{eqnarray*} \label{pa-rep}  f(z) =  \sum_{n =0}^{q-1} \overline{z}^n f_n (z), \ z \in \Oms,
 \end{eqnarray*}
where the functions $f_n$ are analytic on $\Oms.$  Therefore,   $f$ can written in the form 
 \begin{eqnarray*} \label{papol-rep}  f(z) =  \sum_{n =0}^{q-1}  \sum_{m =0}^{+\infty}\overline{z}^n f_m (z), \ z \in \Oms,
 \end{eqnarray*} 
 where $f_m$ are anlytic polynomials and the series  converges uniformly on compact sets of $\Oms $.  In view of Theorem A, we see  that  $f$ admits a unique representation of the 
 \begin{eqnarray*} \label{papol-rep}  f(z) =  \sum_{n =0}^{q-1}  \sum_{m =0}^{+\infty} c_{m, n} H_{m, n}, \ z \in \Oms,
 \end{eqnarray*} 
 where $c_{m, n}$ are complex coefficients and the series  converging uniformly on compact sets of $\Oms. $ 
  Since $f$ is in $L^2(\nu)$ it follows that 
  $$\sum_{n =0}^{q-1}  \sum_{m =0}^{+\infty} |c_{m, n}|^2 < +\infty$$ 
  showing that 
$f\in {\mathcal A}^2(s)$. 
The proof is now complete.}
\qed \end{proofB}

\section{The polyanalytic Bergman space on the unit disc}
In this section apply our approach to different classes of orthogonal polynomials to produces  natural examples of Hilbert spaces of  polyanalytic functions.  

We start with  the weighted  polyanalytic Bergman space on $\D$.  Consider  the weighted   Lebesgue measure on $\D$ given by
 $$dA_{\alpha}(z) := \left(1-|z|^2\right)^{\alpha}\frac{dA(z)}{\pi}, \ \ \alpha > -1,$$ where $dA(z)$ is the Lebesgue measure on $\D.$  We denote by ${\mathcal A}^{\alpha}_{q} (\D)$, the weighted $q$-polyanalytic Bergman space on $\D$ where  $q\in \N$ and  $ \alpha > -1$. This is the space of all $q$-polyanalytic functions $f$ on $\D$ which are square integrable with respect to $dA_{\alpha}(z).$

It can be easily checked that the measure $\nu$ is   the  image measure in $\D$  of $\mu \otimes \sigma $ under the map $(t, \xi) \mapsto \sqrt{t}\xi$ from $[0, 1[\times \T$ onto $\D$ where $\mu$ is the measure $[0, 1[$ given by  $$d\mu(t) := \left(1-t\right)^{\alpha}dt.$$
 The corresponding  moment sequence is
\begin{equation}
s_{d}=\int_{0}^{1}t^{d}\left(1-t\right)^{\alpha}dt =\frac{\Gamma\left(d+1\right)\Gamma\left(\alpha +1\right)}{\Gamma\left(d+\alpha+2\right)}.
\end{equation}

\begin{lemma}\label{tr-rule}  Suppose that $\varphi$ is an automorphism of the unit disc. 
The the Bergman kernel $K_{q, \alpha}$ of ${\mathcal A}^{\alpha}_{q} (\D)$ follows the transformation rule
\begin{eqnarray}  K_{q, \alpha}(z, \xi) = \frac{\left( \varphi'(z) \overline{\varphi'(\xi)}\right)^{(\alpha +q+1)/2}   }{ \left(\overline{\varphi'(z)} \varphi'(\xi)\right)^{(q-1)/2} }K_{q, \alpha}(\varphi(z), \varphi(\xi))
\end{eqnarray} 
for all $z, \xi \in \D.$
 \end{lemma}
\begin{proof} It is sufficient to  assume that $\varphi \circ \varphi (z) = z, $ for all $z\in \D.$ We recall the that the measure 
$\frac{d A(z) }{(1-|z|^2)^2}$ is invariant under the action of the automorphism group of the unit disc.
We also observe that for any fixed $\xi \in \D,$ the function 
$z\mapsto \frac{{(\varphi')(z))}^{(\alpha +q+1)/2}}{{(\overline{\varphi'(z)})}^{(q-1)/2}}  K_{q, \alpha}(\varphi(z), \xi)$ is an element of ${\mathcal A}^{\alpha}_{q} (\D)$.
By the reproducing property and change of variables formula we see that
$$\aligned
\frac{\left( \varphi'(z)\right)^{(\alpha +q+1)/2}}{ \left(\overline{\varphi'(z)}\right)^{(q-1)/2} }K_{q, \alpha}(\varphi(z), \xi)
&= \int_{\D} \frac{\left( \varphi'(w)\right)^{(\alpha +q+1)/2}}{ \left(\overline{\varphi'(w)}\right)^{(q-1)/2} } K_{q, \alpha}(\varphi(w), \xi)K_{q, \alpha}(z,w) dA_{\alpha}(w) \\
&=  \int_{\D}
 \frac{ \left(\overline{\varphi'(w)}\right)^{(\alpha +q+1)/2} }{\left( \varphi'(w)\right)^{(q-1)/2}} K_{q, \alpha}(w,\xi)
 K_{q, \alpha}(z, \varphi(w)) dA_{\alpha}(w) 
\\
&= \overline{ \int_{\D}
 \frac{ \left(\varphi'(w)\right)^{(\alpha +q+1)/2} }{\left(\overline{ \varphi'(w)}\right)^{(q-1)/2}} K_{q, \alpha}(\xi, w)
 K_{q, \alpha}(\varphi(w), z) dA_{\alpha}(w) }  \\
& = 
\frac{\left(\overline{\varphi'(\xi)} \right)^{(\alpha +q+1)/2}}{ \left(\varphi'(\xi)\right)^{(q-1)/2} }K_{q, \alpha}(z, \varphi(\xi))
\endaligned
$$
Replacing $\xi$ by $\varphi(\xi)$ the latter equalities yield
$$\aligned
\frac{\left( \varphi'(z) \overline{\varphi'(\xi)}\right)^{(\alpha +q+1)/2}   }{ \left(\overline{\varphi'(z) } \varphi'(\xi)\right)^{(q-1)/2} }K_{q, \alpha}(\varphi(z), \varphi(\xi))
&= 
K_{q, \alpha}(z, \xi)
\endaligned
$$
This completes the proof.
\end{proof}

 We shall make use of  the classical  Jacobi polynomials $P^{(\alpha, d)}_{n}$ with parameters $(\alpha, d)$ and degree $n.$   An explicit formula for these polynomials is given by

\begin{eqnarray}\label{jacobi}   P^{(\alpha, d)}_{n}(x) =   \frac{1}{2^{n}} \sum_{k=0}^{n} \binom{\alpha + n}{k} 
\binom{d+ n }{ n-k}(x-1)^{ n-k} (x+1)^{k}.
\end{eqnarray}
It is well-known by formula (3.96) in (\cite{STW}, p. 71) that these polynomials verify the equality 
\begin{eqnarray}\label{jacobi}   P^{(\alpha, d)}_{n}(1-2x) =  
 \frac{\Gamma(n+\alpha +1)}{ n!\Gamma(n+\alpha +d+1)}\sum_{j=0}^{n} (-1)^{j} \binom{n}{j} 
  \frac{\Gamma(n+j+\alpha+ d+1)}{ \Gamma(j+\alpha +1)} x^j.
\end{eqnarray}

The Jacobi  polynomials satisfy the orthogonality condition

\begin{eqnarray}\label{orthog-jacobi}   \int_{0}^{1}P^{(\alpha, d)}_{n}(2x-1)P^{(\alpha, d)}_{n'}(2x-1)    x^d (1-x)^\alpha d x =  \delta_{n,n'} h^{\alpha,  d}_{ n} \end{eqnarray}
where
\begin{eqnarray}\label{hmd  }h^{\alpha,  d}_{ n} := \frac{\Gamma\left(\alpha +n+1)  \right)\Gamma\left( d+n+1 \right)}{ \Gamma\left(\alpha +d+n+1  \right)\left( \alpha +d+2n+1 \right)}.
\end{eqnarray}
and hence  for each non-negative integer $d$, the  reproducing kernel of space of polynomials of degree at most $q-1$ with respect to the $L^2$-inner product  associated to the measure $ t^d d\mu(t)$ is then
$$
\aligned
Q_{d,q-1}\left(x, y\right) & =  \sum_{n =0}^{q-1}   \frac{ 
P_{n}^{\alpha, d}(2x-1) P_{n}^{\alpha, d}(2y-1)}{h^{\alpha,  d}_{n}}\\
 & =  \sum_{n =0}^{q-1}   \frac{ 
P_{n}^{d, \alpha}(1-2x) P_{n}^{d, \alpha}(1-2y)}{h^{\alpha,  d}_{n}}.
\endaligned
$$
By the identity (3.114) in (\cite{STW}, p. 75)   we see that 

\begin{eqnarray}\label{Akj}  
Q_{0, q-1}\left(x, 0\right)
   &  =  \frac{\Gamma(q+ \alpha +1)  )}{  (q-1)! \Gamma(\alpha +1 )}
  P_{q-1}^{1, \alpha}(1-2x) 
    \end{eqnarray}
   so that by  (\ref{jacobi}) we obtain
    
 $$\aligned
 F_{q, s} (0,  x, 0 ) & =  Q_{0,q-1}\left(x, 0\right)   \\
 &  =  q  \binom{\alpha + q -1}{\alpha}  \sum_{j=0}^{q-1} (-1)^{j} \binom{q-1}{j} \binom{\alpha +q+j}{\alpha +q-1} 
    x^j 
 \endaligned
$$
We observe that if $z\in \D,$ then $K_{q, \alpha}(z, 0) = F_{q, s} (0,  |z|^2, 0 ).$
Let $z, w\in \D $ let 
$$\varphi_{w}(z) := \frac{z-w}{1-z\bar w}.$$
 By Lemma \ref{tr-rule}, we have
\begin{eqnarray*}  K_{q, \alpha}(z, w) = \frac{\left( \varphi_{w}'(z) 
\overline{\varphi_{w}'(w)}\right)^{(\alpha +q+1)/2}   }{ \left(\overline{\varphi_{w}'(z)} \varphi_w'(w)\right)^{(q-1)/2} }K_{q, \alpha}(\varphi_w(z), 0)
\end{eqnarray*} 
 Since 
 
$$\aligned \frac{\left( \varphi_{w}'(z) \overline{\varphi_{w}'(w)}
\right)^{(\alpha +q+1)/2}   }{ \left(\overline{\varphi_{w}'(z)} \varphi_w'(w)\right)^{(q-1)/2} }
 &= \frac{(1- \bar z w)^{q-1}}{(1- z\bar w)^{\alpha +q+1}}
\endaligned
$$
and

$$\aligned  |\varphi_{w}(z)|^{2j}  = \frac{ [z-w|^{2j} }{ |1 - z\bar w|^{2j}  } 
\endaligned
$$
it follows  that
\begin{eqnarray*}  K_{q, \alpha}(z, w)  & =  
 q  \binom{\alpha + q-1}{\alpha}  \frac{(1- \bar z w)^{q-1}}{(1- z\bar w)^{\alpha +q+1}} 
  \sum_{j=0}^{q-1} (-1)^{j} \binom{q-1}{j} \binom{\alpha +q+j}{\alpha +q-1}    \frac{ |z-w|^{2j} }{ |1 - z\bar w|^{2j}  }.
\end{eqnarray*} 

\section{Weighted polyanalytic Fock spaces }
The second example is the weighted measure defined on $\mathbb{C}$ by
\begin{equation*}
d\nu \left(z\right):=|z|^{2\alpha}  e^{-\left|z\right|^{2}}dA\left(z\right), \ \ \alpha>-1,
\end{equation*}
where $dA\left(z\right)$ is the normalized Lebesgue measure on $\mathbb{C}$. We denote by ${\mathcal A}^{\alpha}_{q} (\C)$ the weighted $q$-polyanalytic Fock space on $\C$ where  $q$ is a positive integer. This is the space of all $q$-analytic functions $f$ on $\C$ which are square integrable with respect to $d\nu(z).$
The measure $\nu$ is the image measure in $\C$ of $\mu \otimes \sigma $ under the 
map $(t, \xi) \mapsto t^{1/2}\xi$ from $[0, +\infty[\times \T$ onto $\C$ where $\mu$ is the measure on $[0, +\infty[$ given by 
 $$d\mu(t) :=  \frac{1}{\Gamma(\alpha +1)} t^{\alpha} e^{-t}dt.$$
 The corresponding  moment sequence is 
\begin{equation}
s_{d}=\int_{0}^{+\infty}t^{d}e^{-t}dt=\frac{\Gamma\left( \alpha +d+1\right)}{\Gamma(\alpha +1) }.
\end{equation}
We will use the classical weighted Laguerre polynomials $L_{n}^{\alpha}$ of degree $n$ and weight $\alpha
$. These polynomials satisfy,
\begin{equation}
\int_{0}^{+\infty}L_{n}^{d+\alpha}\left(x\right)L_{n'}^{d+\alpha}\left(x\right)x^{d+\alpha}e^{-x}dx=\frac{\Gamma\left(d+n+1\right)}{n!}\delta_{n,n'}.
\end{equation}
They have the following explicit representation
\begin{equation} 
L_{n}^{d+\alpha}\left(x\right)  =\sum_{l=0}^{n}\frac{\left(n+d+\alpha\right)!}{r!\Gamma\left(n+d+\alpha +1-r\right)}\frac{\left(-x\right)^{n-r}}{\left(n-r\right)!}.
\end{equation}
We point out some useful formulas, the first one is due to Bailey \cite{Bail}
\begin{equation} \label{bailey}
L_{n}^{d+\alpha}\left(x\right)L_{n}^{d+\alpha}\left(y\right)=\frac{\Gamma\left(d+\alpha+n+1\right)}{n!}\sum_{l=0}^{n}\frac{\left(xy\right)^{n-l}L_{l}^{d+\alpha+2n+-2l}\left(x+y\right)}{\left(n-l\right)!\Gamma\left(d+\alpha+n+1-l\right)}.
\end{equation}
The others are 
\begin{equation}\label{laguerre1}
L_{n}^{d+\alpha}\left(x-y\right)=\sum_{r=0}^{n}\frac{y^{r}}{r!}L_{n-r}^{d+\alpha+r}\left(x\right). 
\end{equation}
\begin{equation}\label{laguerre2}
L_{n}^{d+\alpha}\left(x-y\right)=e^{-y}\sum_{r=0}^{+\infty}\frac{y^{r}}{r!}L_{n}^{d+\alpha+r}\left(x\right)
\end{equation}
which are easy consequences of the definition of $L_{n}^{d+\alpha}\left(x\right)$.
To compute the series $F_{q,s}\left(\lambda,x,y\right)$ in this case, it is sufficient to calculate the following expression
\begin{equation}
S_{\alpha,q}(\lambda)=\sum_{n=0}^{q-1}\sum_{d=-n}^{+\infty}\frac{n!\lambda^{d}}{\Gamma\left(n+d+\alpha +1\right)}L_{n}^{d+\alpha}\left(x\right)L_{n}^{d+\alpha}\left(y\right).
\end{equation}
using first (\ref{bailey}) and then (\ref{laguerre1}) and (\ref{laguerre2}) we have
$$ \aligned
S_{\alpha,q}(\lambda)&=\sum_{n=0}^{q-1}\lambda^{-n}n!\sum_{d=n}^{+\infty}\frac{\lambda^{d+n}L_{n}^{d+\alpha}\left(x\right)L_{n}^{d+\alpha}\left(y\right)}{\Gamma\left(d+\alpha+n +1\right)}
 \\
& =\sum_{n=0}^{q-1}\lambda^{-n}n!\sum_{d=0}^{+\infty}\frac{\lambda^{d}}{\Gamma\left(d+\alpha+1\right)}L_{n}^{d+\alpha-n}\left(x\right)L_{n}^{d+\alpha-n}\left(y\right)
\\
&=\sum_{n=0}^{q-1}\lambda^{-n}\sum_{d=0}^{+\infty}\lambda^{d}\sum_{r=0}^{n}\frac{\left(xy\right)^{n-r}L_{r}^{d+\alpha+n-2r}\left(x+y\right)}{\left(n-r\right)!\Gamma\left(d+\alpha-r+1\right)}
\\
&=\sum_{n=0}^{q-1}\lambda^{-n}\sum_{r=0}^{n}\frac{\left(xy\right)^{n-r}\lambda^{r}}{\left(n-r\right)!}\sum_{d=r}^{+\infty}\frac{\lambda^{d-r}}{\Gamma\left(d+\alpha-r+1\right)!}L_{r}^{d+\alpha+n-2r}\left(x+y\right)
\\
&=\sum_{n=0}^{q-1}\lambda^{-n}\sum_{r=0}^{n}\frac{\left(xy\right)^{n-r}\lambda^{r}}{\left(n-r\right)!}\sum_{d=0}^{+\infty}\frac{\lambda^{d}}{\Gamma\left(d+\alpha+1\right)}L_{r}^{d+\alpha+n-r}\left(x+y\right)
\\
&=\sum_{n=0}^{q-1}\lambda^{-n}\sum_{r=0}^{n}\frac{\left(xy\right)^{n-r}\lambda^{r}}{\left(n-r\right)!}e^{\lambda}L_{r}^{\alpha+ n-r}\left(x+y-\lambda\right)
\\
&=\sum_{n=0}^{q-1}e^{\lambda}L_{n}^{\alpha}\left(x+y-\lambda-\frac{xy}{\lambda}\right)
\endaligned$$
and since $\sum_{r=0}^{q-1}L_{r}^{\beta}=L_{q-1}^{\beta+1}$ we deduce that
\begin{equation}
F_{q,s}\left(\lambda,x,y\right)=e^{\lambda}L_{q-1}^{\alpha+1}\left(x+y-\lambda-\frac{xy}{\lambda}\right).
\end{equation}

\section{The higher dimensional case}
 Consider $p$ Stieltjes moment sequences $s(1), \cdots , s(p) \in  \etSq$ and for each $j$ let $\mu_j \in \cM(s(j))$ and  denote by $\nu_j $ denote the  image measure on $\C$  of $\mu_j \otimes \sigma $ under the map $(t, \xi) \mapsto \sqrt{t}\xi$ from $[0, +\infty[\times \T$ onto $\C$,  where $\sigma$  is the rotation invariant probability measure on the unit circle $\T$ in $\C$.   Then the support of each $\nu_j$ is contained in the closure of the disc $\D_j$ centred at $0$ with radius $R_{s(j)} $.  
  Then  we set
 $\nu := \nu_1 \otimes \cdots \otimes \nu_p$   and
  consider the Hilbert space $L^2(\nu)$ of square integrable
complex-valued functions  in $\overline{\D}_1 \times \cdots \times \overline{\D}_p$ with respect to the measure
$\nu$.
We denote by ${\mathcal A}_{\nu, q}^{2}$   the 
 space  of those $q$-analytic functions on $\D_1 \times \cdots \times \D_p$  which are square integrable with respect to $\nu.$
The natural inner product inherited from that of $L^2(\nu)$ turns ${\mathcal A}^{2}_{\nu, q}$ into a pre-Hilbert space.
  We are now prepared to state the higher dimensional analog of Theorem B.
\newtheorem*{thB'}{\bf Theorem B'}  
\begin{thB'}{\it   The space  ${\mathcal A}_{\nu,q}^2$ is Hilbert space which coincides with the closure  of the $q$-analytic polynomials in  $L^2(\nu)$. Moreover, 
  for each  set compact $K \subset  \Oms$ we have that
 $$ {\sup}_{z \in K} {|f(z)|}
\leq C \parallel f \parallel _{L^2(\nu)}$$
 for all    $q$-analytic  polynomials $f \in L^2(\nu) $, where
 $$C=C(K) := {\sup}_{z\in K}   \prod_{j=1}^n\sqrt{F_{q, s(j)} (|z_j|^{2},  |z_j|^{2}, |z_j|^{2} )}.$$
  Furthermore, 
the reproducing kernel  of   ${\mathcal A}_{\nu}^2$    is given by
$$R_{\nu, q}(z, w) = \prod_{j=1}^p F_{q, s(j)} (z_j\bar w_j, |z_j|^{2}, |w_j|^{2}),  \ \ z, w \in \Oms.$$ }
\end{thB'}

\newtheorem*{RemC'}{\bf Remark C'}   
\begin{RemC'}{\rm    As in Remark C, in the one variable case, when the measures $\mu_j$ have a finite support with exacly $q$ elements, provides the  polyanalytic Cauchy type kernel of the unit 
polydisc $ \D^n:= \{z= (z_1, \cdots, z_n)\in \C :  \max_{j=1, \cdots, n}|z_j| <1\}$. }
\end{RemC'}

In a similar manner, we obtain the weighted polyanalytic Bergman kernel of the unit 
polydisc $\{z= (z_1, \cdots, z_n)\in \C :  \max_{j=1, \cdots, n}|z_j| <1\}$. More precisely,  for $\alpha >-1,$ we consider the  space ${\mathcal A}_{\alpha, n, q}^2$ of all square integrable of  $q$-analytic functions  with respect to the  measure $ d\nu_{\alpha, n}(z) :  \frac{1}{\pi^n} \prod_{j=1}^n(1-|z_j|^2)^\alpha  dV(z), $ where $dV(z)$ is the Lebesgue measure on $\Cn$. We will obtain the following
\newtheorem*{thD'}{\bf Theorem D'}  
\begin{thD'}{\it   The space ${\mathcal A}_{\alpha, n, q}^2$  is Hilbert space which coincides with the closure  of the $q$-analytic polynomials in  $L^2(\nu_\alpha )$ and its  reproducing kernel    is given by
\begin{eqnarray*}  K_{\alpha, n, q}(z, w)  & =  \prod_{j=1}^n K_{\alpha,  q}(z_j, w_j)  
\end{eqnarray*} 
for all $z =(z_1, \cdots, z_n) , w  =(w_1, \cdots, w_n) \in \D^n$, where  $K_{\alpha,  q}$ is the
weighted polyanalytic Bergman kernel of the unit disc. }
\end{thD'}

We also obtain by similar arguments the weighted polyanalytic Bergman kernel for the weighted Fock space in $\Cn$. Namely, let $\alpha >0,$ and denote by
${\mathcal F}_{\alpha, q}(\C)$  the  space ${\mathcal A}_{\alpha, q}^2$ of all square integrable of  $q$-analytic functions  with respect to the  measure $$ d\nu_\alpha(z) : |z|^{2\alpha} e^{- |z|^2} \frac{dV(z)}{\pi^n}, \alpha>-1$$ where $dA(z)$ is the Lebesgue measure on $\Cn$ We will establish the following
\newtheorem*{thE'}{\bf Theorem E'}  
\begin{thE'}{\it   The space ${\mathcal F}_{\alpha, q}$  is Hilbert space which coincides with the closure  of the $q$-analytic polynomials in  $L^2(\nu_\alpha )$ and its  reproducing kernel    is given by
\begin{eqnarray*}  K_{\alpha, n, q}(z, w)  & =  e^{\langle z\bar w \rangle }\prod_{j=1}^n L_{q-1}^{\alpha+1}\left(|z_j- w_j|^2\right)
\end{eqnarray*} 
for all $z=(z_1, \cdots, z_n) , w  =(w_1, \cdots, w_n)\in \Cn$. }
\end{thE'}

\end{document}